\documentclass[english]{article}
\usepackage[T1]{fontenc}
\usepackage[latin9]{inputenc}
\usepackage{graphicx}
\usepackage{amsmath}
\usepackage{amsthm}
\usepackage{amssymb}

\newtheorem{thm}{Theorem}

\makeatletter

\providecommand{\tabularnewline}{\\}

\makeatother

\usepackage{babel}
\begin{document}
\renewcommand{\thefootnote}{\fnsymbol{footnote}}

\title{Dense periodic packings of tetrahedra with small repeating units}

    \footnotetext[1]{Laboratory of Atomic and Solid-State Physics, Cornell University, Ithaca, New York 14853}
    \footnotetext[2]{Department of Genetics, Stanford University School of Medicine, Stanford, California 94305-5120}
\author{
    Yoav Kallus\footnotemark[1], Veit Elser\footnotemark[1], Simon Gravel\footnotemark[2]
    }

\date{}

\maketitle

\begin{abstract}
We present a one-parameter family of periodic packings of regular
tetrahedra, with the packing fraction $100/117\approx0.8547$, that are
simple in the sense that they are transitive and their repeating units
involve only four tetrahedra. The construction of the packings was
inspired from results of a numerical search that yielded a similar
packing. We present an analytic
construction of the packings and a description of their properties.
We also present a transitive packing with a repeating unit of two tetrahedra
and a packing fraction $\frac{139+40\sqrt{10}}{369}\approx0.7194$.
\end{abstract}

\section{Introduction}
The optimization problem of packing regular tetrahedra densely in space
has seen invigorated interest over the last few years \cite{Welsh,Chen,TorqNat,TorqPRE,Glotzer}.
This interest has helped drive up the packing fraction of the densest-known such
packings from $0.7174$ in 2006 \cite{Welsh} to $0.8503$ \cite{Glotzer}
most recently (see Table 3). These improved packing fractions have been obtained
from more and more complex packings, with larger and larger repeating
units. This trend has led some to conjecture that the densest packing
of tetrahedra might have inherent disorder \cite{TorqPRE}. The more
restrictive problem of packing tetrahedra transitively --- that
is, so that all tetrahedra in the packing are equivalent (a more rigorous definition is given
below) --- has been less extensively studied and the densest previously-reported
transitive packing of regular tetrahedra fills only $2/3$ of space \cite{Welsh}.
Here we present a one-parameter family of transitive but dense packings of tetrahedra
with the packing fraction $100/117\approx0.8547$.

The discovery of this family of dense packings was inspired by the
results of a numerical search, which yielded a dense packing with
similar structural properties to the packing we present.
The numerical method used was adapted from the \textit{divide
and concur} approach to constraint satisfaction problems \cite{D-C}.
The \textit{divide and concur} formalism enables us to set up an efficient
search through the parameter space consisting of the positions and
orientations of tetrahedra inside the repeating unit and the translation vectors
governing its lattice repetition, subject to the constraint
that no two tetrahedra overlap. The dynamics involved in the \textit{divide
and concur} search are highly non-physical, which might explain why
our method was able to discover this dense packing, while earlier
methods involving more physical dynamics were not \cite{TorqNat,TorqPRE,Glotzer}.
In this note we present only the analytically constructed packing
without a full explication
of the numerical method, which will be forthcoming.

\section{One-parameter family of dimer-double-lattice packings}

\begin{table}
\begin{tabular}{|l|l|}
\hline 
fundamental & $T_0=\operatorname{conv} \{\mathbf{r}_i \mid i=1,2,3,4\}$
\tabularnewline
tetrahedron &
$\mathbf{r}_{1}=\frac{27}{28}\mathbf{a}-\frac{7}{30}\mathbf{b}+\frac{10}{39}\mathbf{c}$
\tabularnewline
 & 
$\mathbf{r}_{2}=\frac{1}{4}\mathbf{a}-\frac{9}{10}\mathbf{b}$
\tabularnewline
 & $\mathbf{r}_{3}=\frac{1}{14}\mathbf{a}+\frac{1}{10}\mathbf{b}+\frac{5}{13}\mathbf{c}$
\tabularnewline
 & $\mathbf{r}_{4}=\frac{3}{7}\mathbf{a}+\frac{1}{10}\mathbf{b}-\frac{5}{13}\mathbf{c}$
\tabularnewline
\hline 
other tetrahedra & $T_1=R_2(T_0)=\operatorname{conv} \{\frac{2}{3}(\mathbf{r}_2+\mathbf{r}_3+\mathbf{r}_4)-\mathbf{r}_1,\mathbf{r}_2,\mathbf{r}_3,\mathbf{r}_4\}$ \tabularnewline
in the unit cell & $T_2=I(T_0)=-T_0$ \tabularnewline
& $T_3=I(T_1)=-T_1$ \tabularnewline
\hline
space group & translations by $\mathbf{a}$, $\mathbf{b}$, $\mathbf{d}_x=\frac{1}{2}\mathbf{b}+\frac{1}{2}\mathbf{c}+x\mathbf{a}$\tabularnewline
generators & $I=$ inversion about $0$\tabularnewline
& $R_2=$ two-fold rotation about $\{\frac{1}{4}\mathbf{a}+t\mathbf{b} \mid t \in \mathbb{R}\}$\tabularnewline
\hline 
packing fraction & $100/117$\tabularnewline
\hline
coordinate basis for which & $\mathbf{a}=(2\sqrt{7}/5,0,0)$\tabularnewline
tetrahedra are regular & $\mathbf{b}=(0,\sqrt{3}/2,0)$\tabularnewline
with unit edge length & $\mathbf{c}=(0,0,13\sqrt{3/14}/5)$\tabularnewline
\hline
\end{tabular}

\caption{
    The construction of the dimer-double-lattice family of packings
    in terms of the parameter $29/56\le x\le9/14$ in a general monoclinic
    coordinate basis ($\mathbf{a}\cdot\mathbf{b}=\mathbf{b}\cdot\mathbf{c}=0$).
    The packing is generated starting from the fundamental tetrahedron by
    the action of the space group. A packing of regular tetrahedra is
    obtained when the general monoclinic coordinate basis reduces to the
    specified orthogonal coordinate basis.
}

\end{table}

The first set of packings we report on are naturally described as double lattices
of bipyramidal dimers. A double lattice is the union of two Bravais
lattices related to each other by an inversion operation about some
point. In \cite{Kup1}, Kuperberg and Kuperberg used the idea of a
double lattice in the Euclidean plane to show that any planar convex
body can be packed in an arrangement with a packing fraction no smaller
than $\sqrt{3}/2$. We naturally extend the idea of the double lattice
to the three-dimensional Euclidean space. The repeating unit of one
constituent lattice is a bipyramidal dimer: two regular tetrahedra sharing a
common face. Two of these dimers with mutually-inverted orientation ---
a Kuper-pair --- form the repeating unit of the double lattice, which
thus has four tetrahedra of distinct orientations. We state the existence
of the packings as Theorem 1.

\begin{thm}
There exists a one-parameter family of packings of regular tetrahedra,
each having packing fraction $110/117$. These packings are
periodic, with each unit cell of the lattice containing two bipyramidal dimers.
The group of isometries leaving each packing invariant is a crystallographic space group
of type $C2/c$ (following the classification and notation of \cite{spacegroups})
and acts transitively on the individual tetrahedra of the packing.
\end{thm}

\begin{proof} We construct each packing by acting on a single regular tetrahedron
with a group of isometries. For a coordinate basis we use three pair-wise
orthogonal vectors $\mathbf{a}$, $\mathbf{b}$, and $\mathbf{c}$, of norms
$|\mathbf{a}|=2\sqrt{7}/5$, $|\mathbf{b}|=\sqrt{3}/2$, and $|\mathbf{c}|=13\sqrt{3/14}/5$.
The initial tetrahedron $T_0=\operatorname{conv} \{\mathbf{r}_i \mid i=1,2,3,4\}$ 
is the convex hull of four vertices whose coordinates are given in Table 1; 
it is a regular tetrahedron of unit edge length.
The group of isometries is the group of crystallographic type $C2/c$
generated by the translations $\mathbf{a}$, $\mathbf{b}$, 
and $\mathbf{d}_x=\frac{1}{2}\mathbf{b}+\frac{1}{2}\mathbf{c}+x\mathbf{a}$ ($29/56\le x\le9/14$),
by the inversion about the point $0$, and by the rotation by 180 degrees about the axis
$\{\frac{1}{4}\mathbf{a}+t\mathbf{b} \mid t \in \mathbb{R}\}$ \cite{spacegroups}.
This space group has a point group of order 4 and its
translations generate a centered monoclinic point lattice. The construction is
summarized in Table 1.

By construction, each tetrahedron in the packing is the image of $T_0$
under an isometry in the group. We have immediately then that
all tetrahedra are congruent with $T_0$,
that the packing is invariant under the action of the group, and that the group
acts transitively on individual tetrahedra. As the tetrahedra divide into four
orbits of the lattice translations, the packing fraction is given by
$$\phi=\frac{4 \operatorname{vol}(T_0)}{|\det([\mathbf{a},\mathbf{b},\mathbf{d}_x])|}
=\frac{4 (\sqrt{2}/12)}{|\mathbf{a}||\mathbf{b}||\mathbf{c}|/2}
=\frac{100}{117}$$

All that is left then to prove the theorem is to verify that the arrangement of tetrahedra thus constructed
is indeed a packing. As the packing is transitive (in the sense that its symmetry group acts transitively
on the constituent tetrahedra), it is enough to check that one tetrahedron, $T_0$, does not overlap any
other tetrahedron. By means of a closest-lattice-point algorithm such as the one in \cite{LattPt}, we
generate a list of all tetrahedra whose centroid is, for any $29/56\le x\le9/14$, at a distance less than
$\sqrt{3/2}$ from the centroid of $T_0$. There are 46 such tetrahedra. All other tetrahedra have
circumspheres which do not intersect the circumsphere of $T_0$, and therefore
they do not intersect $T_0$. For each tetrahedron in the list we can establish the existence of a separating plane
separating it from $T_0$. Two tetrahedra have no overlap if and only if they are separated by a plane,
and moreover, such a plane always exists which passes through three of the eight vertices of the two tetrahedra. By exhaustively
verifying that one of the finitely many planes that pass through three of the vertices separates the two tetrahedra,
we establish that each tetrahedron in the list can be separated from $T_0$.
\end{proof}

Note that the above construction, which uses a specific orthogonal coordinate basis $\{\mathbf{a},\mathbf{b},\mathbf{c}\}$,
is a special case of a general family of packings of non-regular tetrahedra that can be obtained using
the same construction, but with a general monoclinic coordinate basis ($\mathbf{a}\cdot\mathbf{b}=\mathbf{b}\cdot\mathbf{c}=0$).
Each of these more general packings is an image of a packing in the original family under an affine map from a three-parameter
family (not counting pure dilation). As the lack of overlap
between tetrahedra and the packing fraction are both affine-invariant,
these are also transitive packings of the same packing fraction.

\begin{figure}
\begin{centering}
\includegraphics[bb=10bp 0bp 320bp 200bp,clip,scale=0.5]{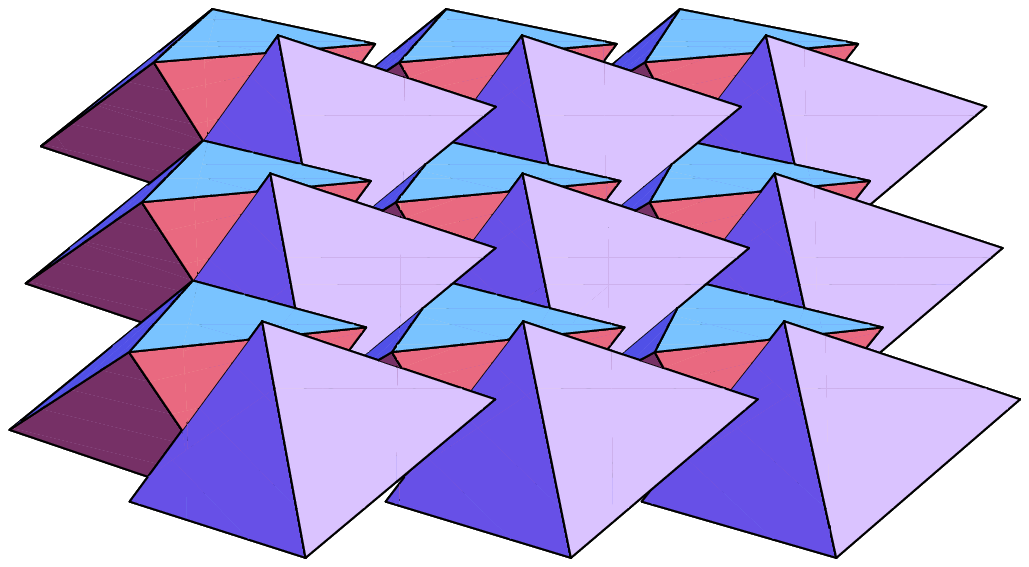}
\includegraphics[bb=20bp 0bp 320bp 240bp,clip,scale=0.6]{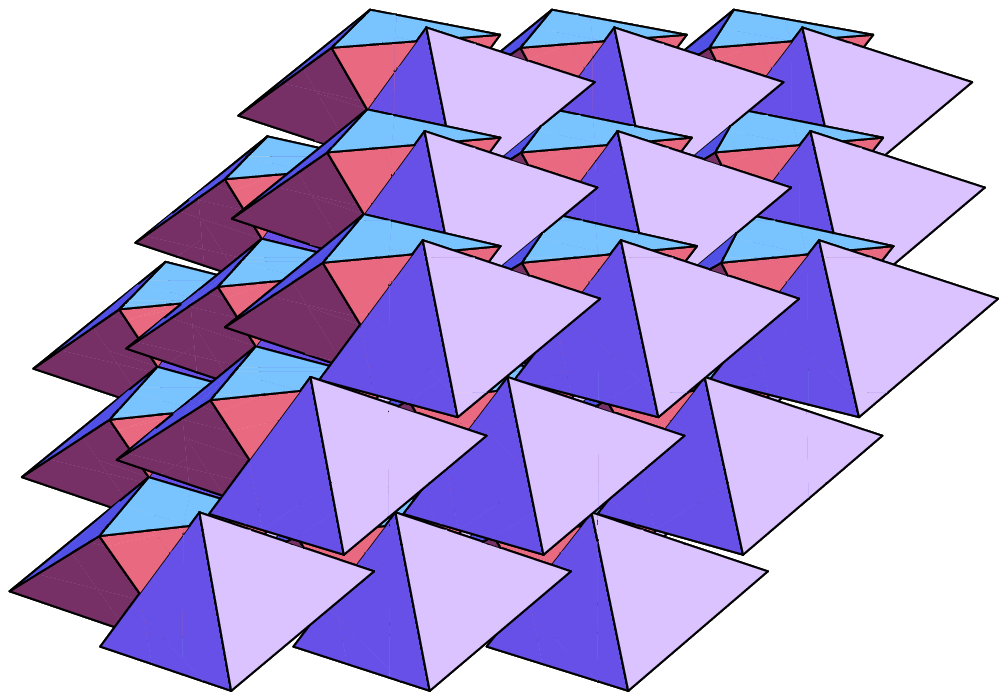}
\caption{Small portions of one layer and three stacked layers
in the dimer-double-lattice packing given by $x=4/7$.}
\end{centering}

\end{figure}
By the construction of the double lattice, there is an inversion center
that sends one lattice of dimers into the other. Note that a lattice
translation composed with an inversion about a point corresponds to
an inversion about a point related to the original inversion center
by half the lattice vector. It follows then that in any primitive unit cell
of the lattice, there are eight such inversion centers. These eight
inversion centers form the vertices of a parallelepiped one-eighth
the volume of the primitive unit cell of the lattice. This parallelepiped
is the equivalent of the ``extensive parallelogram'' described
in \cite{Kup1} whose vertices are the inversion points that generate
the double lattice. As in \cite{Kup1}, the parallelepiped is inscribed
in the body being packed --- the bipyramidal dimer in our case.

An interesting feature of the packing is the presence of
the free parameter $x$. The effect of a change in $x$
is to slide fixed layers of the packing relative one another
along the $\mathbf{a}$-direction. These layers are the layer
generated by the translations $\mathbf{a}$ and $\mathbf{b}$ from
the four tetrahedra of the primitive unit cell and the layers parallel
to it. The construction yields a
valid packing when the small protrusions in one layer
are staggered to fit into small gaps in the neighboring
layer, which is the case for all $29/56\le x\le9/14 \pmod{1}$.
Within this range, each layer can slide against the
neighboring layer without changing the spacing between
the two or creating collisions. It is possible then to
obtain equally dense, non-transitive packings by staggering
consecutive layers arbitrarily within the allowed range.

\begin{figure}
\begin{centering}
\includegraphics[scale=0.55]{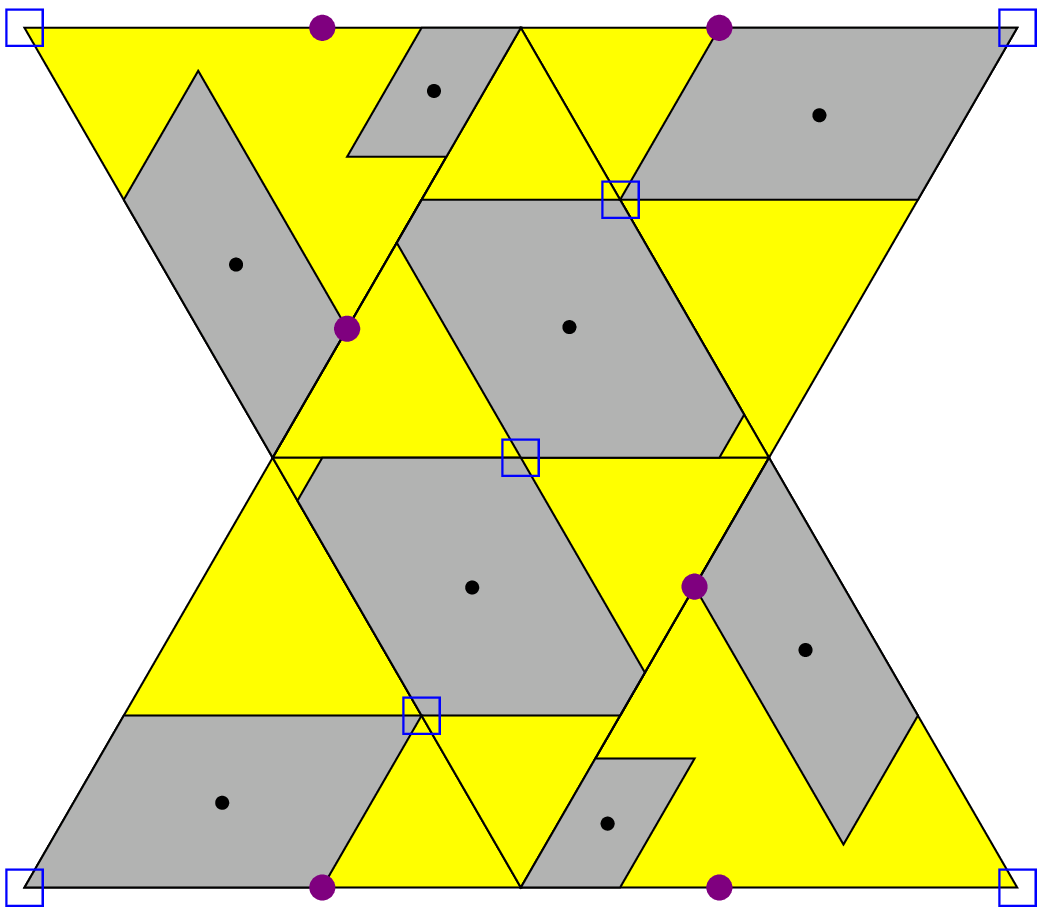}\includegraphics[scale=0.55]{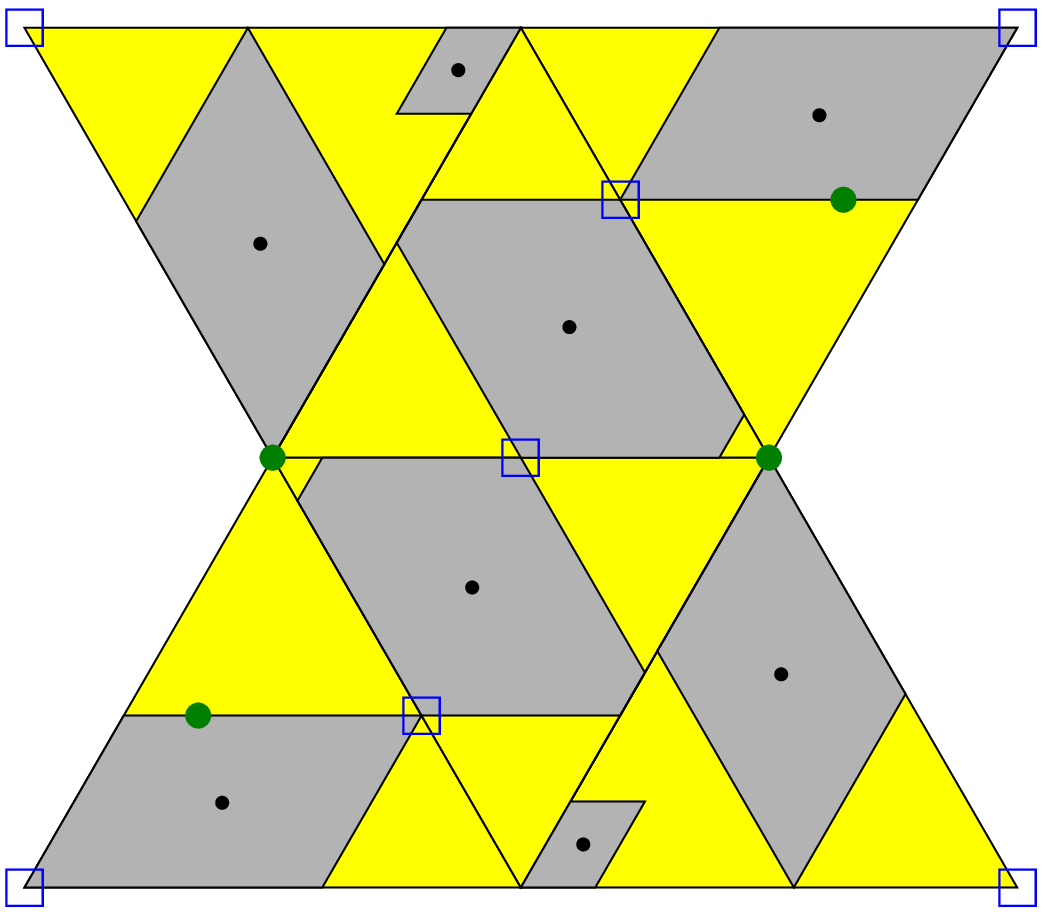}
\par\end{centering}

\caption{The contacts on the surface of a dimer shown on a net diagram for
$x=29/56$ (left) and $x=9/14$ (right): the face-face contacts
(gray), whose centers (black dots) lie on inversion centers, four
of which are fixed and four of which move as a function of $x$;
the four point contacts made regardless of the value of $x$
(blue squares), all lying on two-fold axes; the four point contacts formed
only for $x=29/56$ (purple dots); and the four point contacts formed
only for $x=9/14$ (green dots).}

\end{figure}

We describe next the contacts formed by each dimer in the packing,
and they are illustrated in Figures 1 and 2. Each of the eight vertices
of the inscribed parallelepiped corresponds to the center of a face-face
contact between bipyramids of opposite orientations, accounting for
all contacts between oppositely-oriented bipyramids. Four of these contacts
are within one layer and four of them are with the layers above and below.
The contacts formed between like-oriented bipyramids vary with the parameter $x$:
for all values of $x$ there are two edge-edge contacts, a vertex-edge
contact and an edge-vertex contact (all of these contacts occur on
two-fold axes and are within one layer); for $x=29/56$ there are four
additional edge-edge contacts with dimers in neighboring layers
(which turn into overlaps for $x<29/56$); and for $x=9/14$ there
are instead two vertex-face contacts and two face-vertex contacts with dimers
in neighboring layers (which again turn into overlaps
for $x>9/14$). Thus, each dimer makes respectively twelve, sixteen,
or sixteen contacts in the three cases, and correspondingly, each
tetrahedron makes eight, ten, or eleven contacts.

\begin{table}
\begin{tabular}{|l|l|}
\hline 
fundamental & $T_0=\operatorname{conv} \{\mathbf{r}_i \mid i=1,2,3,4\}$
\tabularnewline
tetrahedron & $\mathbf{r}_{1}=[(433-86\sqrt{10})\mathbf{a}+(611-133\sqrt{10})\mathbf{b}+$
\tabularnewline
&
\hspace{20 pt} $(188-22\sqrt{10})\mathbf{c}]/246$
\tabularnewline
& $\mathbf{r}_{2}=[(111-30\sqrt{10})\mathbf{a}+(93-75\sqrt{10})\mathbf{b}+$
\tabularnewline
& \hspace{20 pt}$(-66-42\sqrt{10})\mathbf{c}]/246$ 
\tabularnewline
 & $\mathbf{r}_{3}=[(-85-28\sqrt{10})\mathbf{a}+(13-29\sqrt{10})\mathbf{b}+$
\tabularnewline
& \hspace{20 pt} $(4+10\sqrt{10})\mathbf{c}]/246$
\tabularnewline
 & $\mathbf{r}_{4}=[(179-106\sqrt{10})\mathbf{a}+(427-101\sqrt{10})\mathbf{b}+$
\tabularnewline
& \hspace{20 pt} $(-20-50\sqrt{10})\mathbf{c}]/246$
\tabularnewline
\hline
other tetrahedron & $T_1=I(T_0)=-T_0$ \tabularnewline
in the unit cell & \tabularnewline
\hline
space group & translations by $\mathbf{a}$, $\mathbf{b}$, $\mathbf{c}$\tabularnewline
generators & $I=$ inversion about $0$\tabularnewline
\hline 
packing fraction & $(139+40\sqrt{10})/369$\tabularnewline
\hline
coordinate basis for & $\mathbf{a}=(1,-(13-4\sqrt{10})/3,0)$\tabularnewline
which tetrahedra are & $\mathbf{b}=(-(4-\sqrt{10})/3,3-\sqrt{10},-1)$\tabularnewline
regular with unit edge& $\mathbf{c}=(3-\sqrt{10},1,(4-\sqrt{10})/3)$\tabularnewline
\hline
\end{tabular}

\caption{The construction of the simple-double-lattice packing in a general triclinic coordinate basis.}

\end{table}

\begin{figure}
\begin{centering}
\includegraphics[bb=10bp 0bp 320bp 300bp,clip,scale=0.4]{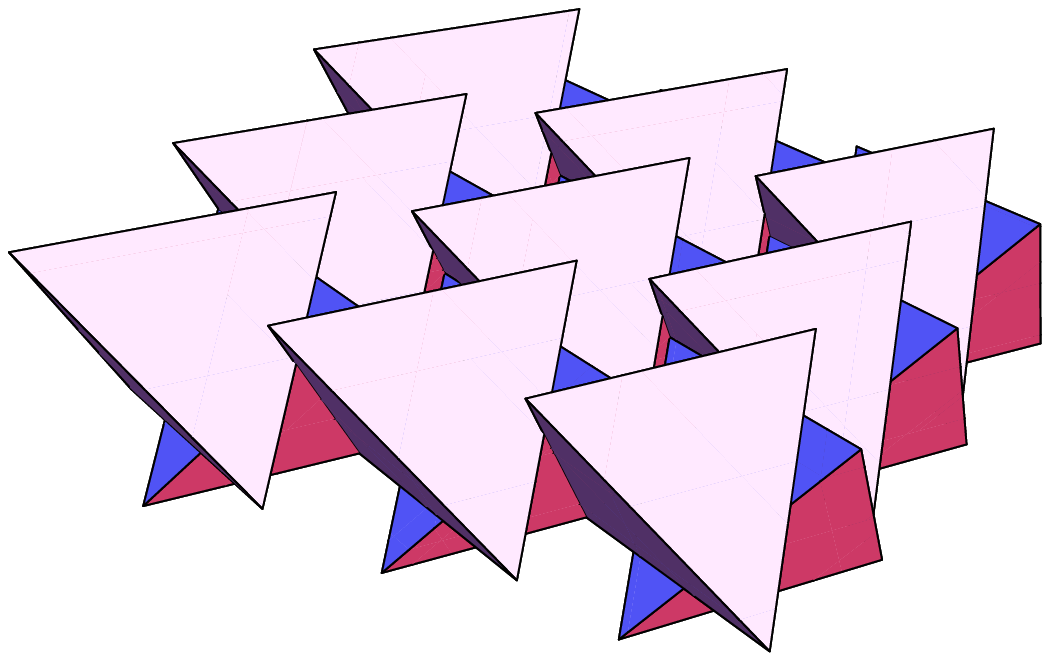}
\includegraphics[bb=20bp 0bp 300bp 350bp,clip,scale=0.4]{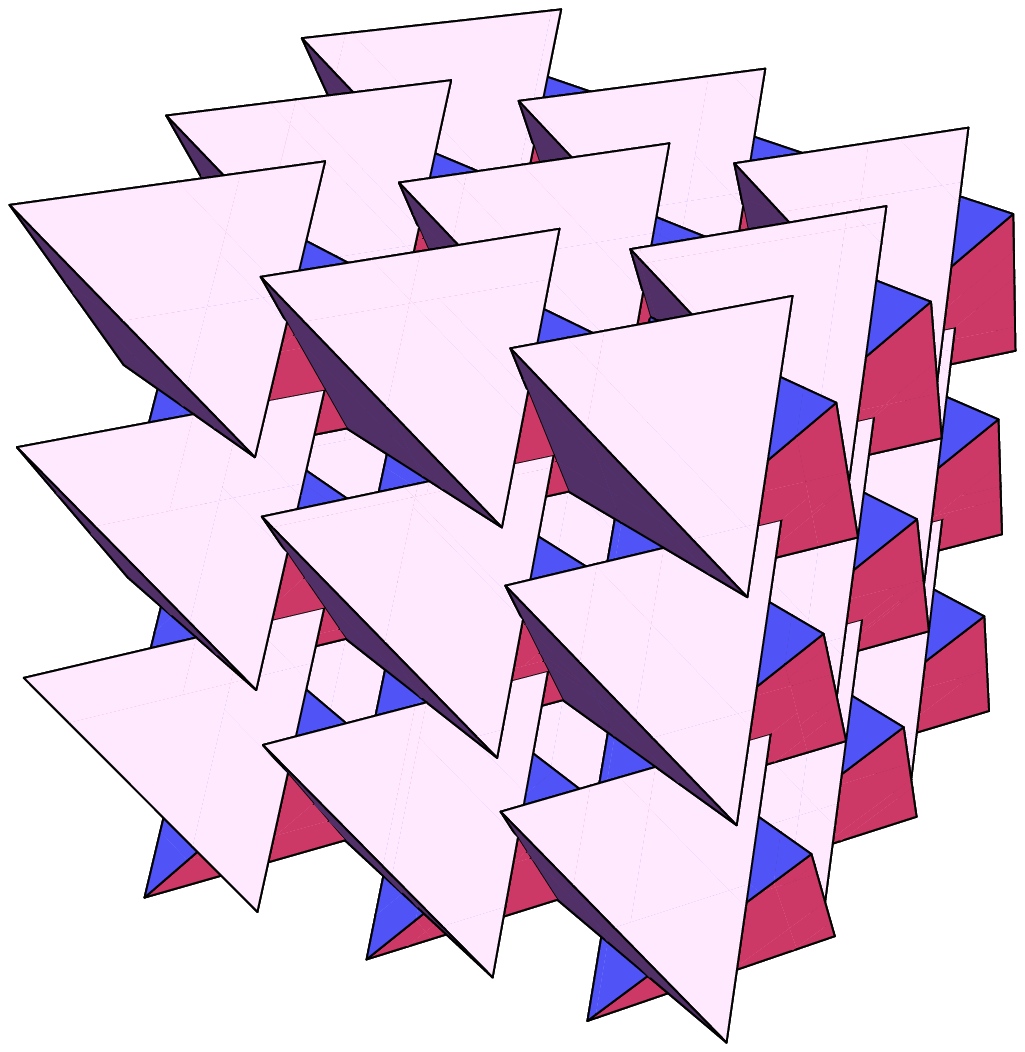}
\caption{Small portions of one layer and three stacked layers
in the simple-double-lattice packing.}
\end{centering}

\end{figure}

\section{Simple double-lattice packing}

Our numerical search also yielded a packing with two tetrahedra
per repeating unit, which we could identify as a simple double lattice
(that is a double lattice of
tetrahedra, not of dimers) with packing fraction
$\frac{139+40\sqrt{10}}{369}\approx0.7194$. We present therefore a second theorem.

\begin{thm}
There exists a packing of regular tetrahedra having packing
fraction $\frac{139+40\sqrt{10}}{369}$. This packing is
periodic, with each unit cell of the lattice containing two tetrahedra.
The group of isometries leaving the packing invariant is a crystallographic space group
of type $P\bar{1}$ and acts transitively on the individual tetrahedra of the packing.
\end{thm}

The proof of this theorem proceeds equivalently to the proof of Theorem 1. The
vertices of the initial tetrahedron $T_0$ are given in Table 2 in terms of a triclinic coordinate
basis which is also given in Table 2. The group of isometries we use to construct the
packing is generated by translation by the three vectors $\mathbf{a}$, $\mathbf{b}$, and $\mathbf{c}$ and by inversion
about the origin. This is a space group of crystallographic type $P\bar{1}$, with a point group
of order 2 and a triclinic lattice \cite{spacegroups}. Figure 3 shows a portion of the packing.

The simple-double-lattice packing again
has eight inversion centers per primitive cell at
the vertices of a parallelepiped. However, in this case only five of
the vertices are on the surface of the tetrahedron. Each tetrahedron
in the packing is in contact with nineteen others.

\begin{table}
\begin{tabular}{lllll}
Name & $\phi$ & $N$ & $\bar{Z}$ & Transitive\tabularnewline
\hline 
Optimal lattice\cite{Latt} & $18/49\approx0.3673$ & $1$ & $14$ & Yes\tabularnewline 
Warp and weft\cite{Welsh} & $2/3\approx0.6666$ & $2$ & $10$ & Yes\tabularnewline 
Welsh\cite{Welsh} & $17/24\approx0.7083$ & $34$ & $25.9$ & No\tabularnewline 
Simple double lattice & $\frac{139+40\sqrt{10}}{369}\approx0.7194$ & $2$ & $19$ & Yes\tabularnewline 
Wagon wheels\cite{Chen} & $0.7786$ & $18$ & $7.1$ & No\tabularnewline 
Compressed wagon wheels\cite{TorqNat} & $0.7820$ & $72$ & $7.6$ & No \tabularnewline 
Disordered wagon wheels\cite{TorqPRE} & $0.8226$ & $314$ & $7.4$ & No\tabularnewline 
Quasicrystal approximant\cite{Glotzer} & $0.8503$ & $656$ & & No \tabularnewline
Dimer double lattice & $100/117\approx0.8547$ & $4$ & $8$ to $11$ & Yes\tabularnewline 
\hline 
\end{tabular}

\caption{Some studied transitive and non-transitive packings of regular tetrahedra with
packing fraction $\phi$, number of tetrahedra in the repeating unit $N$, and
average number of contacts per tetrahedron $\bar{Z}$ where available.}

\end{table}

\section{Discussion}

In Table 3, we compare the packings presented here to other studied
packings of regular tetrahedra. Both packings are denser than the densest
previously-reported transitive packing, a double lattice presented by
Conway and Torquato (which we call the "warp-and-weft"
packing due to the interweaving arrangement of its tetrahedra) \cite{Welsh},
and the dimer double lattice is denser than any previously-reported
packing.

The results presented go against the recent trend of ever-growing
repeating units in densest-known packings and demonstrate that a
large repeating unit is not a necessary property
of a dense packing of regular tetrahedra. It is curious that previous
simulations, utilizing a more physical search
dynamic \cite{TorqNat,TorqPRE, Glotzer}, yielded dense packings that
were either disordered, had quasicrytalline order, or
had crystalline order characterized by a very large repeating unit,
and were not able to find the denser class of structures presented here,
(reminiscent perhaps of Kurt Vonnegut's ice-nine,
a fictional phase of water that is more stable, but kinetically unreachable).

Our results yield the surprising situation wherein the densest-known
packing of icosahedra is now sparser than the corresponding packing
of tetrahedra, a solid which just four years ago was a prime candidate
for a counterexample of a conjecture by Ulam that the sphere is
the sparsest-packing convex solid \cite{Welsh}. As the packing
can be generally extended to any tetrahedron in a three-parameter family
generated by deformations of the monoclinic coordinate basis,
if any tetrahedron provides a counterexample
of Ulam's conjecture, it is not a tetrahedron of that family.

The regular tetrahedron is no longer outcast, as it long was,
from the respectable family of convex polyhedra whose largest-achieved packing
density is realized by a transitive arrangement.
While there are some convex solids whose maximum packing density clearly
cannot be achieved by a transitive arrangement (the convex
Schmitt-Conway-Danzer polyhedron can tile space, but only in
aperiodic and non-transitive ways \cite{Danz}), the majority
of regular and semi-regular polyhedra have been to-date packed most
densely in transitive packings \cite{TorqNat}.
Whether this situation is accidental, the result of bias favoring
the discovery of transitive packings, or a more fundamental property
governing the packing of a certain class of solids is still an open question.

This work was supported by grant NSF-DMR-0426568.

\end{document}